\documentclass[reqno]{amsart}
\usepackage{amsfonts}
\usepackage{amsmath}
\usepackage{amsthm}
\usepackage{amssymb}
\usepackage{amscd}
\usepackage{mathtools}
\usepackage{mathrsfs}
\usepackage[latin1,utf8]{inputenc}
\usepackage{fancyhdr}
\usepackage{cite}

\usepackage{graphicx}
\usepackage{float}
\usepackage[dvipsnames]{xcolor}
\usepackage{pgf,tikz,pgfplots}
\usetikzlibrary{arrows}
\pgfplotsset{compat=1.15}

\usepackage[numbers]{natbib}

\usepackage{hyperref}
\hypersetup{
    colorlinks=true,
    linkcolor=Violet,
    citecolor=Violet,
    urlcolor=Violet,
}

\makeatletter
\newcommand{\setword}[2]{%
	\phantomsection
	#1\def\@currentlabel{\unexpanded{#1}}\label{#2}%
}
\makeatother

\definecolor{ccqqww}{rgb}{0.8,0.,0.4}
\definecolor{ffccww}{rgb}{1.,0.8,0.4}
\definecolor{qqwwzz}{rgb}{0.,0.4,0.6}
\definecolor{ttzzqq}{rgb}{0.2,0.7,0.4}

\theoremstyle{plain}
\newtheorem{theorem}{Theorem}[section]
\newtheorem{lemma}[theorem]{Lemma}
\newtheorem{proposition}[theorem]{Proposition}
\newtheorem{proposition*}{Proposition}
\newtheorem{corollary}[theorem]{Corollary}

\theoremstyle{definition}

\newtheorem{example}{Example}

\theoremstyle{remark}
\newtheorem{remark}{Remark}

\numberwithin{equation}{section}

\DeclareMathOperator{\R}{\mathbb{R}}

\DeclareMathOperator{\C}{\mathbb{C}}

\renewcommand{\H}{\mathbb{H}}
\renewcommand{\S}{\mathbb{S}}
\renewcommand{\a}{\mathfrak{a}}
\renewcommand{\k}{\mathfrak{k}}

\DeclareMathOperator{\I}{\mathcal{I}}

\DeclareMathOperator{\Z}{\mathbb{Z}}
\DeclareMathOperator{\casimir}{\mathcal{C}}

\DeclareMathOperator{\g}{\mathfrak{g}}

\DeclareMathOperator{\m}{\mathfrak{m}}

\DeclareMathOperator{\divergence}{\text{\!div}}

\DeclareMathOperator{\End}{\text{End}}

\DeclareMathOperator{\Rep}{\text{Rep}}
\DeclareMathOperator{\id}{\text{id}}

\DeclareMathOperator{\Ad}{\text{Ad}}

\DeclareMathOperator{\rank}{\text{rank}}

\DeclareMathOperator{\Isot}{\text{Isot}}

\DeclareMathOperator{\res}{\text{res}}
\DeclareMathOperator{\sym}{\text{sym}}

\DeclareMathOperator{\U}{\mathcal{U}}

\DeclareMathOperator{\llangle}{\langle\langle}
\DeclareMathOperator{\rrangle}{\rangle\rangle}

\title[A note about the generic irreducibility of the spectrum of the Laplacian on Homogeneous Spaces]{A note about the generic irreducibility of the spectrum of the Laplacian on Homogeneous Spaces}

\author[D. S. de Oliveira \and
M. A. M. Marrocos
]
{Diego S. de Oliveira \and
Marcus A. M. Marrocos
}

\address{Centro de Matem\'atica, Computa\c{c}\~ao e Cogni\c{c}\~ao, Universidade Federal do ABC\\
Av. dos Estados, 5001\\
09210-580 Santo Andr\'e, S\~ao Paulo\\
Brazil.}
\email{diego.sousa@ufabc.edu.br}
\address{Departamento de Matemática, ICE, Universidade Federal do Amazonas\\
Av. General Rodrigo Octávio Jordão Ramos, 6200, Campus Universitário, Coroado I,\\
Manaus, AM, 69080-005.}

\email{marcusmarrocos@ufam.edu.br}
\thanks{The first author was supported in part by grant, Coordenação de Aperfeiçoamento de Pessoal de Nível Superior (CAPES). The  second author acknowledge support from Fundação de Amparo à Pesquisa do Estado do Amazonas (FAPEAM)}



\begin{document}

\maketitle



\pagestyle{fancy} 
\fancyhf[LHE]{\scalebox{0.6}{DIEGO S. DE OLIVEIRA AND MARCUS A. M. MARROCOS}}
\fancyhf[RHE]{\thepage}
\fancyhf[LHO]{\thepage}
\fancyhf[RHO]{\scalebox{0.5}{A NOTE ABOUT THE GENERIC IRREDUCIBILITY OF THE
SPECTRUM OF THE LAPLACIAN ON HOMOGENEOUS
SPACES}}
\fancyhf[CHE]{}
\fancyhf[CHO]{}

\fancyhf[LFE]{ }
\fancyhf[RFE]{ }
\fancyhf[LFO]{ }
\fancyhf[RFO]{ }
\fancyhf[CFE]{ }
\fancyhf[CFO]{ }



\begin{abstract}
Petrecca and Röser  (2018, \cite{Petrecca2019}), and Schueth (2017, \cite{Schueth2017}) had shown that for a generic $G$-invariant metric $g$ on certain compact homogeneous spaces $M=G/K$ (including symmetric spaces of rank 1 and some Lie groups), the spectrum of the Laplace-Beltrami operator $\Delta_g$ was real $G$-simple. The same is not true for the complex version of $\Delta_g$ when there is a presence of representations of complex or quaternionic type. We show that these types of representations induces a $Q_8$-action that commutes with the Laplacian in such way that $G$-properties of the real version of the operator have to be understood as $(Q_8 \times G)$-properties on its corresponding complex version. Also we argue that for symmetric spaces on rank $\geq 2$ there are algebraic symmetries on the corresponding root systems which relates distinct irreducible representations on the same eigenspace.  
\end{abstract}

\setcounter{tocdepth}{1}
\tableofcontents

\section{Introduction}

\,

\,

\textbf{Motivation}

\,

The study of the Laplacian appears in many applications as diffusion phenomenons or state levels of energy.
In differential geometry, this study is intimately related to the calculus of order 2 on manifolds and to the manifestation of the curvature notion. Laplacians and their spectra can allow us to produce some explicit differential constructions, which are helpful in order to improve our knowledge and qualitative information on the given manifold.

The Laplace-Beltrami operator (or simply Laplacian) acts on the space of smooth functions $C^\infty(M)$, on a compact connected Riemannian manifold $(M,g)$, by
\[
\Delta_g := - \divergence (\,\text{grad} \;\cdot \,)
\]
and its eigenvalues are distributed in the line $\R$ by a sequence
\[
\lambda_0 = 0 < \lambda_1 \leq \lambda_2 \leq \cdots \to \infty \;.
\]
This operator also commutes with isometries by pullback. Conversely, every diffeomorphism that commutes with $\Delta_g$ by pullback is, in fact, an isometry with respect to $g$.  Thus, the Laplacian measures, in some sense, the presence of geometric symmetries on the manifold. Here, we understand the notion of ``symmetry" by any map that impacts on the multiplicities of the Laplacian's eigenvalues.  Uhlenbeck (1976, \cite{Uhlenbeck1976}) showed that, for a generic metric $g$ on the manifold $M$, the Laplace-Beltrami operator has simple spectrum, that is, for a compact connected manifold most of the metrics have not nontrivial geometric symmetries, even it could have be some isometries which do not play any role on the multiplicities.

We restrict our analysis to the class of the connected compact homogeneous spaces of the form $M=G/K$, which allow us to apply some algebraic facts, converting a great part of the constructions on abstract manifolds to the machinery of linear algebra and representation theory, which can be more tangible. More specifically, in some contexts, the complex version of the Laplacian $\Delta_g$ can be identified as algebraic operator in the form
\[  
 \bigoplus_{V \in \widehat{G}_K} \Delta_g^{V^K} \otimes \id : \bigoplus_{V \in \widehat{G}_K} V^K \otimes V^* \to \bigoplus_{V \in \widehat{G}_K} V^K \otimes V^* \; ,
\]
where $\widehat{G}_K$ is a complete set of nonequivalent spherical irreducible representations with respect to the pair $(G,K)$, and $\Delta_g^{V^K}: V^K \to V^K$ is obtained as a restriction of the an linear operator on the $G$-module $(\rho,V)$ given by
\[
 \Delta_g^V = -\sum\limits_{j=1}^N (\,\rho_*(Y_j)\,)^2 : V \to V \; ,
\]
for a $g$-orthonormal basis $\{Y_j\}_{j=1}^N$ of $\g$ (see more details in Subsection \ref{preliminaries:laplace-operators}). Therefore, in this setting, the spectrum of $\Delta_g$ remounts to the study of the characteristic polynomials of the operators $\Delta_g^{V^K}$.  

Suppose that we are only considering $G$-invariant metrics on $M=G/K$. Here $G$ can be understood as a prescribed group of geometric symmetries. In this context, the eigenspaces of the Laplacian are $G$-modules. Now, we can set an analogous problem to that studied by Uhlenbeck. In this new problem, we want to decide if the eigenspaces of the Laplacian $\Delta_g$ are $G$-irreducible representations for a generic $G$-invariant metric $g$, meaning that their corresponding multiplicities are affect only by the prescribed symmetries. That lead us to the notion of $G$-simple spectrum. The spectrum of an operator is called real (resp. complex) $G$-simple if its eigenspaces are real (resp. complex) irreducible representations of $G$.

Some advances were obtained about the spectral problem with prescribed symmetries by a given group $G$. Schueth (2017, \cite{Schueth2017}) proved that, for some compact Lie groups $G$, a generic left invariant metric satisfies $\Delta_g$ with real $G$-simple  spectrum. Similarly, Petrecca and Röser (2018, \cite{Petrecca2019}) proved that, for compact symmetric spaces $M=G/K$ of rank $1$, $\Delta_g$ has a real $G$-simple spectrum for a generic $G$-invariant metric, and they also proved that the same result is not true for rank $\geq 2$.

\,

\textbf{Real $G$-properties vs.\ complex $(Q_8 \times G)$-properties of the spectrum.}

\,

The Laplacians can be studied over $\R$ or $\C$ (in the last case, by an elementary complexification). The classical problems in analysis and differential geometry usually consider the operators over $\R$. But, when the problem is closely related to the representation theory, it is convenient to present them over $\C$, in order to make a better use of the algebraic machinery of this theory. The Laplace-Beltrami operator $\Delta_g : L^2(M,\R) \to L^2(M,\R)$ has a complex version $\Delta_g : L^2(M,\C) \to L^2(M,\C)$.  

It is convenient to establish how to migrate from the constructions over $\R$ to the constructions over $\C$. When the only representations that appears in the complex spectral problem are of real type, it can be show that the property of real $G$-simple spectrum for $\Delta_g$ corresponds to a complex $G$-simple spectrum. In the presence of any  representation of complex or quaternionic type, it turned out that the finite group of quaternions $Q_8$ plays an important role in the spectral setting and, in this case, a real $G$-simple spectrum corresponds to a complex $(Q_8\times G)$-simple spectrum. The main idea for the last case is that, when we migrate from the real version of the $\Delta_g$ to its complex version, we automatically earn some structural algebraic $Q_8$-maps that commutes with $\Delta_g$. Therefore, in the presence of  representations of complex or quaternionic type, $G$-properties for the real operator $\Delta_g$ must be translated as $(Q_8\times G)$-properties for its corresponding complex version. We make this analysis on Section \ref{sec:Q8-section}. 

\,

\,

\textbf{Structural symmetries for symmetric spaces of rank $\geq 2$}

\,

Although Petrecca and Röser (2018, \cite{Petrecca2019}) had established that symmetric spaces with rank $\geq 2$ do not satisfy the irreducibility criterium for the spectrum of the Laplace-Beltrami operator, there is a reason behind it. 

The $G$-invariance of the metric is itself a strong condition referring to prescribed geometric symmetries, which can lead to other structural symmetries (of algebraic nature, for example). In the case of a compact symmetric pair $(G,K)$ (thus, a spherical pair), some geometric conditions can produce algebraic properties on the corresponding root system. In rank $\geq 2$, we gain structural algebraic symmetries on the root system which are responsible to relate irreducible representations with the same Casimir eigenvalue. These different irreducible representations with the same Casimir eigenvalue contributes to the same eigenspace, making impossible the $G$-simplicity of the spectrum. However, we still can establish a generic panorama.

After we take in account all the geometric and structural symmetries on the corresponding root system, we will be able to show a generic situation for the eigenspaces of the Laplacian described by them. Section \ref{sec:rank2-section} is dedicated to this goal.

\,

\,


\section{Preliminaries}

\,

Unless we make other mentions, we consider $M = G/K$ as a compact connected $m$-dimensional homogeneous space, where $G$ is a $N$-dimensional connected, compact and simple Lie group acting by left translations on $M$ and $K$ is a 
 $k$-dimensional closed subgroup of $G$. We fix $g_0$ a bi-invariant metric on $G$ that induces on the Lie algebra level the $g_0$-orthogonal splitting $\g = \k \oplus \m$, where $\m$ is the isotropy $K$-representation. We take $g$ as an arbitrary $(G\times K)$-invariant metric on $G$ (i.e. invariant by $G$-left translations and $K$-right translations), which corresponds to a $G$-invariant metric on $M$ with the same notation $g$. $\widehat{G}$ denotes a complete set of nonequivalent complex irreducible representations of the Lie group $G$. $\mathcal{U}(\g^{\C})$ is the universal enveloping algebra of $\g$.

We begin with some notation and basic definitions. First, we define the induced $G$-module of the trivial $K$-module $\C$ by
\[
C^\infty(G,K;\C) := \{ f \in C^{\infty}(G,\C) \; | \; \forall x \in G, \, \forall k\in K, f(xk)=f(x)   \} \, ,
\]
which is endowed with the left action
\[
L(x)f := f \circ \ell_{x^{-1}} \, , \; \forall x \in G, \; \forall f \in C^\infty(G,K;\C),
\]
where $\ell_{x^{-1}}:G\ni y \mapsto x^{-1}y\in G$ is the left translation by $x^{-1}$. Let
\[
\langle f_1 , f_2 \rangle_{L^2} := \int_G = f_1(x)  \overline{f_2(x)} dx \, , \; \forall f_1,f_2 \in C^\infty(G,K;\C)
\]
where $\int_G (\cdot) dx$ is the Haar integral of $G$ with unitary volume. So, we have the $L^2$-completion of $C^\infty(G,K;\C)$ denoted by $L^2(G,K;\C)$. This new space has a Peter-Weyl decomposition given by the Hilbert sum
\begin{equation} \label{eq:peter-weyl-general-decomposition}
    \bigoplus_{V \in \widehat{G}_{K}} \Isot(V^*) \simeq \bigoplus_{V \in \widehat{G}_{K}} V^K \otimes V^* \, ,
\end{equation}
where $\widehat{G}_{K}:= \{ (\rho,V)\in \widehat{G} \; | \, \dim V^K >0 \}$ and $\Isot(V^*)$ denotes the $V^*$-isotypical component of $L^2(G,K;\C)$ with respect to the $L$-action, endowed with an isomorphism
\begin{equation} \label{eq:isomorphisms-isotypical-general-decomposition}
    \begin{array}{c}
     \varphi_{V^*}: V^K \otimes V^* \ni v \otimes \xi \mapsto \phi_{v,\xi} \in   \Isot(V^*) \; , \\
    \text{where $\phi_{v,\xi} := (\rho^*(\cdot)^{-1}\xi)(v) = \xi(\rho(\cdot)v) $ ,}
\end{array}
\end{equation}
defined by linear extension. 

Our main goal is to study operators $\Delta_g : L^2(G,K;\C) \to L^2(G,K;\C)$ which commutes with the $(G\times K)$-isometries with respect to the metric $g$, by describing them inside each isotypal component in the decomposition \ref{eq:peter-weyl-general-decomposition} and in terms of the isomorphisms \ref{eq:isomorphisms-isotypical-general-decomposition}.

Define $\Isot(V,\overline{V}) := \Isot(V)+\Isot(\overline{V}) \subset L^2(G,K;\C) $. Since $V \simeq \overline{V}$ for $V$ of real or quaternionic type, and $V \ncong \overline{V}$ for $V$ of complex type (as $G$-modules) then 

\begin{equation} \label{eq:isotipica}
\Isot(V,\overline{V}) = 
\begin{cases}
\Isot(V) + \Isot(\overline{V}) \text{, for $V$ of complex type,}\\
\Isot(V) \text{, for $V$ of real or quaternionic type.}
\end{cases}
\end{equation}
\,


\subsection{Types of each representation}
\label{sec:types-representation}
We recall that for each complex finite dimensional $G$-module $V$, we have $V^* \simeq \overline{V}$, where $\overline{V}$ is the conjugate representation of $V$, that is, $\overline{V}$ coincides with $V$ as sets but any complex scalar $z$ acts as $\overline{z}$ in the scalar multiplication of $\overline{V}$. $V$ is called of \emph{real type} if it admits a linear $G$-map $J_V:V \to \overline{V}$ such that $J_V^2 = \id$. $V$ is called of \emph{quaternionic type} if it admits a linear $G$-map $J_V:V \to \overline{V}$ such that $J_V^2 = -\id$. In both cases, $J_V$ is called a \emph{structural map}.  $V$ is called of \emph{complex type} if $V \ncong \overline{V}$ as $G$-modules, i.e. it is not of real type neither of quaternionic type. 
For $V$ of complex type, we have that $V \oplus \overline{V} \simeq \H \otimes V$ is of quaternionic type (for more details, see \cite[Sec.II-6]{Brocker1985}). 

The representations $V$ of quaternionic type are identified as representations in the category $\Rep(G,\H)$ (i.e. with scalars on $\H$) and, under this identification, the structural map $J_V$ is simply the left multiplication by the quaternionic element $j$. 

Finally, we mention that for each $V \in \widehat{G}$ there exists a irreducible real $G$-representation $V_{\R}$ such that 
\begin{equation*} 
    \C \otimes V_{\R} \simeq \left\{ \begin{array}{l}
     \text{$V$, if $V$ is of real type} \\
     \text{$\H\otimes V$, if $V$ is of complex or quaternionic type}
\end{array} \right. 
\end{equation*}
(see \cite[Sec.1]{Petrecca2019} or \cite[Sec.2]{Schueth2017}).

\,


\subsection{Laplace operators} \label{preliminaries:laplace-operators}

We assume that $\{Y_j\}_{j=1}^N$ is a $g$-orthonormal basis of $\g$ such that $\{Y_j\}_{j=k+1}^N$ is a orthonormal basis of $\m$ and we define the left-invariant operator of order 2
\[
\casimir(Y_1,\dots,Y_N):= -\sum\limits_{j=1}^N Y_j^2 : C^{\infty}(G,\C) \to C^{\infty}(G,\C)  \; .
\]
For an arbitrary $G$-representation $(\rho,V)$ (finite dimensional or not), with induced representation $\rho_*:\U(\g^{\C}) \to \End(V)$,  we consider
\[
\Delta_g^V := \rho_*(\casimir(Y_1,\dots,Y_N)) = -\sum\limits_{j=1}^N \rho_*(Y_j^2) : V \to V \; .
\]
In particular, we have the operator
\[
\Delta_g:= L_*(\casimir(Y_1,\dots,Y_N)) = -\sum\limits_{j=1}^N L_*(Y_j^2) : C^\infty(G,K;\C) \to C^\infty(G,K;\C) \; .
\]
(and we have its extension $\Delta_g: L^2(G,K;\C) \to L^2(G,K;\C)$).

\begin{remark}
If we follow the construction on \cite[Sec.1]{Petrecca2019}, then we note that $\Delta_g$ coincides with the Laplace-Beltrami operator associated to $(M,g)$, up to an identification. 
\end{remark}

\begin{remark} \label{remark:eigenspaces-j-invariant}
By \cite[Sec.1]{Petrecca2019}, if $V$ is a $G$-module of quaternionic type with structural map $J_V: V \to \overline{V}$, then the eigenspaces of $\Delta_g^V$ are $J_V$-invariant.
\end{remark}

\begin{remark} \label{remark:laplacian-laplacianV-tensor-id}
 Also by \cite[Sec.1]{Petrecca2019}, if we apply the isomorphism $\Isot(V^*) \simeq V^K \otimes V^*$, then the Laplacian can be identified as
\begin{equation}  \label{eq:laplacianV-tensor-id-tensor-id}
\Delta_g \simeq \bigoplus_{V \in \widehat{G}_K} \Delta_g^{V^K} \otimes \id : \bigoplus_{V \in \widehat{G}_K} V^K \otimes V^* \to \bigoplus_{V \in \widehat{G}_K} V^K \otimes V^* \; ,
\end{equation}
where $\Delta_g^{V^K}:= \Delta_g^V|_{V^K}: V^K \to V^K$. It also holds that 
\begin{equation*}
    \Delta_g^{V^K} = -\sum\limits_{j=k+1}^N \rho_*(Y_j^2) : V^K \to V^K \; .
\end{equation*}    
\end{remark}

\,


\subsection{Basic facts on root systems of symmetric pairs} \label{preliminaries:root-systems}

Let $(\a, R_{\g,\k})$ the corresponding root system for a symmetric pair $(G,K)$, endowed with inner product $(\cdot,\cdot)$ (which is unique up to a scalar). The half sum of positive roots is $\delta$ and the Weyl group is $W$. We also fix a Weyl Chamber $C$ and define the lattices
\[
\Gamma_G := \{ H \in \a \, ; \; \exp_G(2\pi H) = e \}
\]
and
\[
\I := \sum\limits_{\alpha \in R_{\g,\k}} \Z \cdot \alpha^\vee \; ,
\]
where $\alpha^\vee := \frac{2}{\langle \alpha, \alpha \rangle } \alpha$. We recall that in the simply connected case, we have $\Gamma_G^* = \I^*$.

So, the spherical representations (i.e. the $G$-irreducible modules $V$ such that $\dim V^K >0$) are in one-to-one correspondence with the set $C \cap \Gamma_G^*$. Therefore, each spherical representation is in the form $V^\mu$, indexed by its highest restriced weight $\mu \in C\cap \Gamma_G^*$ and it has Casimir eigenvalue
\begin{equation*}
    \begin{array}{l}
    \lambda_\mu := a_\mu^2 - (\delta,\delta) \; ,  \\
     \text{where $a_\mu := ( \mu + \delta,\mu+\delta)^{1/2} $ .}
\end{array}
\end{equation*}
For more details, see \cite{Hall2015, Helgason2001}.

\begin{remark} \label{remark:same-casimir-eigenvalue}
    Note that two spherical representations $V^\mu$ and $V^\eta$, with respect to the pair $(G,K)$, have the same Casimir eigenvalue if, and only if, $a_\mu = a_\eta$. The last claim is the same to say that $\mu$ and $\eta$ must be in the same sphere centered in $-\delta$.
\end{remark}

\,


\subsection{Basic facts on the quaternion group $Q_8$} \label{preliminaries:Q8} 

It is well known that the finite quaternion group
\[
Q_8 := \{ \pm 1, \pm i, \pm j, \pm ij \}
\]
has only one $2$-dimensional complex irreducible representation which is isomorphic to faithful representation $\H$ (the ring of quaternions endowed with the standard $Q_8$-action by left multiplication). The other complex irreducible representations are of degree $1$.

\,


\subsection{The space of $G$-invariant metrics} \label{preliminaries:G-invariant-metrics}
We know that the $G$-invariant metrics on $M = G/K$ are in one-to-one correspondence with $(G \times K)$-invariant metrics on $G$. There are other characterizations that we will see now. First, let
\begin{equation*}
    \begin{array}{l}
        \sym_K(\m):=  \{ \kappa \in \End(\m) ; \; \text{$\kappa$ is symmetric and $\Ad_K$-equivariant}  \} \, , \\
        \sym_K^+(\m):= \{ \kappa \in \sym_K(\m)  ; \; \text{each eigenvalue of $\kappa$ is positive}  \} \, .
   \end{array}
\end{equation*}
Note that $\sym_K(\m)$ has an induced Euclidean topology from $\End(\m)$. The same is true for $\sym_K^+(\m)$.

For each $\kappa \in \sym_K^+(\m)$, we define 
\[
g_\kappa := g_0(\kappa^{-1}(\cdot) \, , \, \cdot):\m \times \m \to \R
\]
which corresponds to a $G$-invariant metric on $M$. Therefore, the set of $(G\times K)$-invariant metrics on $G$ can be identified as the space $\sym_K^+(\m)$, which turned out to be an open subset of $\sym_K(\m)$ (see more on \cite[Sec.1]{Petrecca2019}).

For our purposes, a residual subset in a topological space is a subset that can be expressed as a countable intersection of open dense subsets in this space. In this context,  the $G$-invariant metrics on $M$ that corresponds to a residual subset of $\sym_K^+(\m)$ are called \emph{generic}. 

So, in this identification, the operators on Subsection \ref{preliminaries:laplace-operators} are actually indexed by a parameter $\kappa \in \sym_K^+(\m)$. We can generalize these constructions to operators with a parameter $\kappa \in \sym_K(\m)$ (a more general space). We do this in the following subsection. 

\,


\subsection{Operators on each representation}
\label{prelimiaries:operators-on-each-representation}
Fix a $g_0$-orthonormal basis $\{Y_j\}_{j=1}^N$ of $\g$ such that $\{Y_j\}_{j=k+1}^N$ is an orthonormal basis of $\m$ and let $(\rho,V) \in \widehat{G}$.

Now, we can define for each $\kappa = (\kappa_{ij}) \in \sym_K(\m)$ the operators
\begin{equation*}
    D^{V}(\kappa) := - \sum_{i,j = 1}^N \kappa_{ij} \, \rho_*(Y_i \cdot Y_j) : V \to V  \; .
\end{equation*}
Note that, when $k \in \sym_K^+(\m)$ corresponds to a metric $g$, then $D^V(\kappa) = \Delta_g^V$ and in this case we can write $ \Delta_\kappa^V := \Delta_g^V = D^V(\kappa)$. Therefore, the family of operators $\{ D^V(\kappa) \}_{\kappa \in \sym_K(\m)}$ is bigger than the family $\{ \Delta_\kappa^V \}_{\kappa \in \sym_K^+(\m)}$.

For each $\kappa \in \sym_K(\m)$, we can also define
\begin{equation*}
    D^{V^K}(\kappa) = -\sum\limits_{i, j=k+1}^N \kappa_{ij} \, \rho_*(Y_i \cdot Y_j) : V^K \to V^K \; .
\end{equation*}
Similarly, when $g$ corresponds to a $\kappa \in \sym_K^+(\m)$, we have $\Delta_{\kappa}^{V^K} = D^{V^K}(\kappa)$.

We note that $D^{V^K}(\kappa) = D^V(\kappa)|_{V^K}:V^K \to V^K$. For details on the precedent constructions, we refer to \cite[Sec.1]{Petrecca2019}.

\,


\subsection{Recent advances in the literature} \label{preliminaries:recent-advances-results-literature}
Now, we mention some important advances presented by \cite{Petrecca2019} and \cite{Schueth2017}.  Let $V, V_1,V_2 \in \widehat{G}$ and $\kappa \in \sym_K(\m)$.

There exists a map $\res:\C[t]\times \C[t] \to \C$, called \emph{resultant}, that satisfies for each pair of polynomials $p,q \in \C[t]$: (i) $\res$ is a polynomial map such that $\res(p,q)$ is a polynomial on the coefficients of $p$ and $q$; (ii) $\res(p,q)$ is the zero polynomial if and only if $p$ and $q$ share any common zero. This map is useful because it allow us to compare the eigenvalues of the operators $D^{V^K}(\kappa)$ (or $D^V(\kappa)$) by taking their corresponding characteristic polynomials. Define
\begin{equation*}
     \begin{array}{ccl}
        p_{V}(\kappa) &:=& \text{charac.polynomial}(D^{V^K}(\kappa)) \, ,\\
        a_{V_1,V_2} (\kappa) &:=& \res\left(p_{V_1}(\kappa), p_{V_2}(\kappa)\right) \, , \\
        b_{V}(\kappa) &:=& \res\left(p_{V}(\kappa), \frac{d}{dt} p_{V}(\kappa)\right) \, , \\
        c_{V}(\kappa) &:=& \res\left(p_{V}(\kappa), \frac{d^2}{dt^2} p_{V}(\kappa)\right)  .
    \end{array}
\end{equation*}
Thus we have polynomial maps
\[
a_{V_1,V_2},\,b_V,\,c_V: \sym_K(\m) \to \C
\]
and the following theorem:
\begin{theorem}
\small \label{teo:criterio-metrica-G-simples-generica} 
There exists a $G$-invariant metric $g$ such that the Laplace-Beltrami operator $\Delta_g$ on $M$ has a real $G$-simple spectrum if and only if the items bellow are simultaneously satisfied:

      \begin{enumerate}
          \item For all $V_1,V_2 \in \widehat{G}_K$, with $V_1 \ncong V_2$ and $V_1 \ncong V_2^*$, $a_{V_1,V_2}$ is not the zero polynomial.

          \item For all $V\in \widehat{G}_K$ of real or complex type, $b_{V}$ is not the zero polynomial.

          \item For all $V\in \widehat{G}_K$ of quaternionic type, $c_{V}$ is not the zero polynomial.
      \end{enumerate}
Moreover, the existence of a such metric is equivalent to say that the Laplace-Beltrami operator of a generic $G$-invariant metric on $M$ on $M$ has real $G$-simple spectrum.
\end{theorem}

\begin{example} \label{ex:Schueth-SU2}
    Let $G = (SU(2) \times \cdots \times SU(2) \times T^n) / \Gamma $, where $\Gamma$ is a discrete central subgroup. Then for a generic left invariant metric $g$ on $G$, the Laplace-Beltrami operator $\Delta_g$ is real $G$-simple. This result was proven by Schueth \cite{Schueth2017} in 2017.
\end{example}

When $M = G/K$ is a compact irreducible symmetric space with metric $g$, then for each irreducible $G$-representation the operator $\Delta_g^V: V \to V$ is the Casimir element on $V$.  Thus, each operator $\Delta_g^{V^K} = \Delta_g^V|_{V^K}:V^K \to V^K$ is a multiple of the identity. Therefore, in this panorama, $\Delta_g^{V^K}$ has simple spectrum if and only if $\dim_{\C} V^K = 1$.

We say that $K$ is a \emph{spherical subgroup} of $G$ if, for each $V\in \widehat{G}_K$, we have that $\dim_{\C} V^K = 1$. Thus, as $K$ is a spherical subgroup, each operator $\Delta_g^{V^K}$ has a single eigenvalue and $\Isot(V^*) \simeq V^K \otimes V^* \simeq  \C \otimes V^* \simeq V^*$. So, the verification of the Theorem \ref{teo:criterio-metrica-G-simples-generica}, in this context, reduces to the verification of the item (1), that is, we only need to verify that distinct isotypical components do not produce common eigenvalues. It can be shown that for every irreducible symmetric pair $(G,K)$, the subgroup $K$ is spherical with respect to $G$. However we warn that not every subgroup $K$ of $G$ is spherical for an arbitrary homogeneous space $M=G/K$ (not even for a normal homogeneous space).

\begin{example} \label{ex:Petrecca-Symmetric-Spaces}
    Let $M$ a compact irreducible symmetric space. Then the Laplace-Beltrami operator is real $G$-simple if and only if $\rank(M) = 1$. Also, if $M$ is a product of compact rank one symmetric spaces, then the Laplace-Beltrami operator is real $G$-simple. In 2018, Petrecca and Röser \cite{Petrecca2019} argued that for $\rank(M) = 1$ there is only one spherical irreducible representation for each Casimir eigenvalue and that representation must be of real type, while for $\rank(M) \geq 2$ we can construct two or more non-equivalent and non-dual spherical representations with the same Casimir eigenvalue.
\end{example}

\,

\,


\section{Real $G$-properties vs.\ complex $(Q_8 \times G)$-properties of the spectrum} \label{sec:Q8-section}

\,

Now, consider the following constructions and identifications:

\begin{itemize} \small
    \item[(J1)] For $V \in \widehat{G}_K$ of real type, the real structural map $j:= J_V$ satisfies $j^2 = \id_{V}$ and it induces a map $j\otimes \id: V^K\otimes \overline{V} \to V^K \otimes \overline{V}$.

    \item[(J2)] For $V \in \widehat{G}_K$ of quaternionic type, the quaternionic structural map $j:=J_V$ satisfies $j^2 = -\id_{V}$ and it induces a map $j\otimes \id: V^K\otimes \overline{V} \to V^K \otimes \overline{V}$.

    \item[(J3)] For $V \in \widehat{G}_K$ of complex type, the $G$-module $(V \oplus \overline{V})$ admits a quaternionic structural map $j: (V \oplus \overline{V}) \to (V \oplus \overline{V})$, which induces a map  
    \[
    j\otimes \id: (V^K\otimes \overline{V})\oplus(\overline{V}^K\otimes V) \to (V^K\otimes \overline{V})\oplus(\overline{V}^K\otimes V) \; .
    \]
\end{itemize}

The maps $j\otimes \id$ on (J1), (J2) and (J3) can be extended to a map  
\[
J : \bigoplus_{V \in \widehat{G}_K} V^K \otimes V^* \to \bigoplus_{V \in \widehat{G}_K} V^K \otimes V^* \; ,
\]
which commutes with the operator $\Delta_g$ given by Equation \ref{eq:laplacianV-tensor-id-tensor-id} (this commutativity follows from Remark \ref{remark:eigenspaces-j-invariant}).

Therefore the operator $\Delta_g$ commutes with the $Q_8$-group identified as a group of automorphisms on $L^2(G,K;\C)$ given by
\begin{equation} \label{eq:Q8-group}
    Q_8 := \{ \pm \id, \pm i \id, \pm J, \pm iJ \} \, .
\end{equation}

Also, it is well known that $\H \otimes V \simeq V \oplus \overline{V}$. Thus, each $G$-submodule isomorphic to $\H \otimes V \simeq V \oplus \overline{V} \hookrightarrow L^2(G,K;\C)$, for a $G$-irreducible submodule $V$, must be a $(Q_8 \times G)$-irreducible submodule.

\begin{remark} \label{remark:translating-G-properties-to-Q8xG-properties}
  We recall that irreducible representations of quaternionic or complex type $V$ satisfy that $V\oplus \overline{V}$ is the complexification of an unique real irreducible representation $V_{\R}$. Thus, the analysis of the complex version of $\Delta_g$ on $\Isot(V,\overline{V})$ reduces to the analysis of $\Isot(V_{\R})$ on its corresponding real version. Because of that, we do not see any effect of the $Q_8$-group \ref{eq:Q8-group} in the real version of the operator, even in the presence of representations of complex or quaternionic type (case in that $Q_8$ plays an important role in the spectral decomposition for the complex version of the Laplacian $\Delta_g$).
\end{remark}

Now we state the main theorem of this section.

\begin{theorem} \label{th:Q8xG-simples-vs-G-simples-real} Let $M=G/K$ a compact homogeneous space and $g$ a $G$-invariant metric. Consider $\Delta_g$ the Laplace-Beltrami operator acting on $M$. Then, $\Delta_g$ has a real $G$-simple spectrum if and only if its corresponding complex version, also denoted by $\Delta_g$, has a complex $(Q_8 \times G)$-simple spectrum. In particular, the real version of Laplace-Beltrami operator on $M$ is generically real $G$-simple if and only if its complex version is generically complex $(Q_8 \times G)$-simple.
\end{theorem}

\begin{proof}
     Let $g$ be a $G$-invariant metric on $M$ and $V \in \widehat{G}_K$. Take $\lambda$ as an eigenvalue of $\Delta_g^{V^K}$ with multiplicity $m(\lambda)$. Denote by $V^K_\lambda$ and $\Isot(V,\overline{V})_\lambda$ the $\lambda$-eigenspaces of $\Delta_g^{V^K}$ and $\Delta_g|_{\Isot(V,\overline{V})}$, respectively. In particular, $m(\lambda) = \dim_{\C}V^K_\lambda$. 
     
     By equations \ref{eq:isotipica} and  \ref{eq:laplacianV-tensor-id-tensor-id}, if $V$ is of real type or of quaternionic type, then we know the restricted operator $\Delta_g|_{\Isot(V,\overline{V})}$ is identified (up to isomorphisms in the domain and in the codomain) as 
     \[\Delta_g^{V^K} \otimes \id_{\overline{V}}: V^K \otimes \overline{V}  \to V^K \otimes \overline{V} \, ,\] whose $\lambda$-eigenspace is given by $V^K_\lambda \otimes \overline{V} \simeq \overline{V}^{\oplus m(\lambda)}$. Since $V \simeq \overline{V}$ for $V$ of real or quaternionic type, then $V^K_\lambda \otimes \overline{V}$ is isomorphic to the space $V^{\oplus m(\lambda)}$. Recall Subsection \ref{sec:types-representation}. Since $\H \otimes V \simeq V\oplus \overline{V}$ and $V \simeq \overline{V}$ for $V$ of real or quaternionic type, then we also have the isomorphism $(\H \otimes V)^{m(\lambda)/2} \simeq V^{\otimes m(\lambda)}$ whenever $m(\lambda)$ is even.

     When $V$ is of complex type, we have a similar construction. In this case, equations \ref{eq:isotipica} and  \ref{eq:laplacianV-tensor-id-tensor-id} state that $\Delta_g|_{\Isot(V,\overline{V})}$ can be identified as the direct sum of the operators $\Delta_g^{V^K} \otimes \id_{\overline{V}}$ and $\Delta_g^{\overline{V}^K} \otimes \id_{V}$. Therefore, the $\lambda$-eigenspace of $\Delta_g|_{\Isot(V,\overline{V})}$ is isomorphic to $(V^K_\lambda \otimes \overline{V}) \oplus (\overline{V^K_\lambda} \otimes V) \simeq (\overline{V}^{\oplus m(\lambda)}) \oplus (V^{\oplus m(\lambda)}) \simeq (V \oplus \overline{V})^{\oplus m(\lambda)}$. Again, by Subsection \ref{sec:types-representation}, we can use the isomorphism $\H \otimes V \simeq V \oplus \overline{V}$ in order to conclude that the $\lambda$-eigenspace of $\Delta_g|_{\Isot(V,\overline{V})}$ is isomorphic to $(\H \otimes V)^{m(\lambda)}$.

     By summarizing the last two paragraphs, the $\lambda$-eigenspace of $\Delta_g|_{\Isot(V, \overline{V})}$ is given by: 

     \begin{equation} \label{I}
    \Isot(V,\overline{V})_{\lambda} \simeq \left\{ \begin{array}{l}
     \text{$ V^{\oplus m(\lambda)} \simeq (\C \otimes V^{\oplus m(\lambda)})$, if $V$ is of real type}, \\
     \text{$(\H\otimes V)^{\oplus m(\lambda)}$, if $V$ is of complex type}, \\ \text{$(\H\otimes V)^{\oplus \,m(\lambda)/2}$, if $V$ is of quaternionic type}
    \end{array} \right.  
    \end{equation}

Subsection \ref{preliminaries:Q8} ensures us that $\H$ is an irreducible $Q_8$-module and since $V$ is an irreducible $G$-module, then the space $\H \otimes V$ must be an irreducible $(Q_8 \times G)$-module.

By Subsection \ref{sec:types-representation}, we can consider $V_{\R}$ a irreducible real representation such that  
\begin{equation*}
    \C \otimes V_{\R} \simeq \left\{ \begin{array}{l}
     \text{$V$, if $V$ is of real type,} \\
     \text{$\H\otimes V$, if $V$ is of complex or quaternionic type.}
\end{array} \right. \; .
\end{equation*}
Thus,
\begin{equation*}
    \Isot(V,\overline{V})_\lambda \simeq \left\{ \begin{array}{l}
     \text{$ (\C \otimes V_{\R})^{\oplus m(\lambda)}$, if $V$ is of real type,} \\
     \text{$(\C \otimes V_{\R})^{\oplus m(\lambda)}$, if $V$ is of complex type,} \\ \text{$(\C \otimes V_{\R})^{\oplus \,m(\lambda)/2}$, if $V$ is of quaternionic type.}
\end{array} \right. 
\end{equation*}
Then the $\lambda$-eigenspace of the restriction of $\Delta_{g}$ to the space $\Isot(V,\overline{V})\cap C^\infty(G,K;\R)$ is given by

\begin{equation} \label{II}
     \Isot(V_{\R})_\lambda \simeq \left\{ \begin{array}{l}
     \text{$V_{\R}^{\oplus m(\lambda)}$, if $V$ is of real type,} \\
     \text{$V_{\R}^{\oplus m(\lambda)}$, if $V$ is of complex type,} \\ \text{$ V_{\R}^{\oplus \,(m(\lambda)/2)}$, if $V$ is of quaternionic type.}
\end{array} \right. 
\end{equation}

By equations \ref{I} and \ref{II}, if $V$ is of real or complex type, then $\Delta_{g}|_{\Isot(V,\overline{V})}$ has complex $(Q_8 \times G)$-simple spectrum if and only if $\Delta_g^V$ has simple spectrum if and only if $\Delta_{g} |_{\Isot(V_{\R})}$ has real $G$-simple spectrum. Similarly, if $V$ is of quaternionic type, then $\Delta_{g}|_{\Isot(V,\overline{V})}$ is complex $(Q_8 \times G)$-simple if and only if each eigenvalue of $\Delta_g^V$ has multiplicity $2$ if and only if $\Delta_{g}|_{\Isot(V_{\R})}$ is real $G$-simple. By collecting up all isotypical components together, we conclude that the complex version of $\Delta_{g} $ has complex $(Q_8 \times G)$-simple spectrum if and only if the real version of $\Delta_{g}$ has real $G$-simple spectrum.
\end{proof}

\begin{remark}
    As an immediate consequence of the proof of Theorem \ref{th:Q8xG-simples-vs-G-simples-real}, if we suppose that $\widehat{G}_K$ has only representations of real type on the compact homogeneous space $M=G/K$, we can also conclude that the Laplace-Beltrami operator has a real $G$-simple for a generic $G$-invariant metric on $M$ if and only if its complex version has a complex $G$-simple spectrum for a generic $G$-invariant metric on $M$.
\end{remark}

Theorem \ref{th:Q8xG-simples-vs-G-simples-real} also show us that each example considered by \cite{Petrecca2019} and \cite{Schueth2017} in which we have a real $G$-simple spectrum for $\Delta_g$ satisfies that the corresponding complex version of $\Delta_g$ has a complex $(Q_8 \times G)$-simple spectrum, as in the following examples:

\begin{example} \label{ex:Lie-groups}
Let $G = (SU(2) \times \cdots \times SU(2) \times T^n)/\Gamma$, where $T^n$ is the $n$-torus and $\Gamma$ is any central discrete subgroup. Then, by Example \ref{ex:Schueth-SU2} and Theorem \ref{th:Q8xG-simples-vs-G-simples-real}, the complex operator $\Delta_g$ is $(Q_8 \times G)$-simple for a generic left invariant metric $g$.
\end{example}

\begin{example} \label{ex:symmetric-spaces}
Let $M = G/K$ a compact symmetric space. Example \ref{ex:Petrecca-Symmetric-Spaces} ensure us that if $\rank(M) =1$ or $M$ is a product of compact rank $1$ symmetric spaces, then each representation on $\widehat{G}_K$ is of real type. Thus the complex operator $\Delta_g$ has a complex $G$-simple spectrum for a generic $G$-invariant metric $g$, in such way that $Q_8$ does not play any role here. Also, we have that $\Delta_g$ has not a generic complex $(Q_8 \times G)$-simple spectrum for $\rank(M) \geq 2$.
\end{example}

Example \ref{ex:symmetric-spaces} only shows us that even the bigger group $Q_8 \times G$, containing $G$, still can not organize the eigenspaces of $\Delta_g$ as complex irreducible representations for rank $\geq 2$. There are some explanations for this effect related to some especial symmetries on the root systems. We explore them in the next section.

\,

\,


\section{Structural symmetries for symmetric spaces of any rank} \label{sec:rank2-section}

\,

When $M = G/K$ is a irreducible symmetric space with metric $g$, then for each irreducible $G$-representation the operator $\Delta_g^V: V \to V$ is the Casimir element on $V$.  Thus, each operator $\Delta_g^{V^K} = \Delta_g^V|_{V^K}:V^K \to V^K$ is a multiple of the identity. Therefore, in this panorama, $\Delta_g^{V^K}$ has simple spectrum if and only if $\dim_{\C} V^K = 1$.

We say that $K$ is a \emph{spherical subgroup} of $G$ if, for each $V\in \widehat{G}_K$, we have that $\dim_{\C} V^K = 1$. Thus, when $K$ is a spherical subgroup, each operator $\Delta_g^{V^K}$ has a single eigenvalue and $\Isot(V^*) \simeq V^K \otimes V^* \simeq  \C \otimes V^* \simeq V^*$. So, the verification of Theorem \ref{teo:criterio-metrica-G-simples-generica}, in this context, reduces to the verification of the item (1), that is, we only need to verify that distinct isotypical components do not produces common eigenvalues. However we warn that not every subgroup $K$ of $G$ is spherical.

Assume that $M = G/K$ is a symmetric space with metric $g$. In general, we have $\Isot(V^*) \simeq V^K \otimes V^*$. Since $\Delta_g^{V^K}$ has only one eigenvalue $\lambda = \lambda_V$, then 
\[
\Isot(V^*) = \Isot(V^*)_{\lambda} \simeq V^K \otimes V^* \simeq (V^*)^{\oplus \dim_{\C} V^K} \, .
\]
Now, we want to compare $\Isot(V_1^*)$ and $\Isot(V_2^*)$ for distinct elements $V_1,V_2 \in \widehat{G}_K$. Note that if $V_1$ and $V_2$ produces the same Casimir eigenvalue $\lambda$, then both of them contributes with the same $\lambda$-eigenspace of $\Delta_g$. We want to describe precisely which representations have $\lambda$ as their Casimir eigenvalue and we want to show that they are related by some algebraic symmetries. 

Let $V$ an irreducible $G$-module. We denote by $\lambda_V$ the Casimir eigenvalue of $\Delta_g^V$. Note that, if $\lambda$ is a eigenvalue of $\Delta_g$, then its corresponding $\lambda$-eigenspace $E_\lambda$ is given by
\begin{equation} \label{eq:lambda-eigenspace}
    L^2(G,K;\C)_\lambda \simeq \bigoplus\limits_{ \begin{array}{c}
         V \in \widehat{G}_K  \\
         \lambda_V = \lambda
    \end{array} } V^K \otimes V^*  \simeq \bigoplus\limits_{ \begin{array}{c}
         V \in \widehat{G}_K  \\
         \lambda_V = \lambda
    \end{array} } (V^*)^{\oplus \dim_{\C} V^K} \, .
\end{equation}

If $K$ is a spherical subgroup of $G$ (every compact symmetric pair $(G,K)$ satisfies this condition), we have
\begin{equation} \label{eq:lambda-eigenspace-spherical-subgroup}
    L^2(G,K;\C)_\lambda \simeq \bigoplus\limits_{ \begin{array}{c}
         V \in \widehat{G}_K  \\
         \lambda_V = \lambda
    \end{array} } V^*\, .
\end{equation}

\,


\subsection{Root systems and representations with the same Casimir eigenvalue} \label{subsec:Root-systems-and-representations-with-the-same-Casimir-eigenvalue}

 Recall the definitions and constructions on Subsection \ref{preliminaries:root-systems} for the root system $(\a, R_{\g,\k}, (\cdot \, , \, \cdot) \,)$ of the symmetric pair $(G,K)$, with metric $g$, Weyl chamber $C$, Weyl group $W$, lattice $\Gamma_G := \{ H \in \a \, ; \; \exp_G(2\pi H) = e \}$ and half sum of positive roots $\delta$. The inner product $(\cdot \, , \, \cdot)$ on $\a$ is induced by $g$.  We assume that $\delta \in \Gamma^*_G$ throughout this section (this occurs in the simply connected case, where $\Gamma_G = \I$, for example). We also assume that we are in the irreducible case, that is, we suppose that $\a$ is an irreducible root system.  We know that every $V \in \widehat{G}_K$ corresponds to an element $\mu \in C \cap \Gamma_G^*$ and we denote this representation by $V^\mu$. Thus, $\Delta_g^{V^\mu}= \lambda_\mu  \id$, where the scalar $\lambda_\mu$ is given by 
\begin{equation} \label{eq:casimir-eigenvalue}
    \begin{array}{l}
    \lambda_\mu := a_\mu^2 - (\delta,\delta) \; ,  \\
     \text{where $a_\mu := ( \mu + \delta,\mu+\delta)^{1/2} $ .}
\end{array}
\end{equation}
Equation \ref{eq:casimir-eigenvalue} and Remark \ref{remark:same-casimir-eigenvalue} lead us to the following proposition:
\begin{proposition}  \label{prop:autovalores-casimir-mesma-esfera-centrada-em--delta}
  Two representations $V^\mu, V^\eta \in \widehat{G}_K$ have the same Casimir eigenvalue $\lambda = \lambda_\mu = \lambda_\eta$ if and only if $\mu, \eta \in \S_{a}(-\delta) $, for some $a\geq 0$.  In this case, $a = a_\mu = a_\eta $ and $\lambda = a^2 - (\delta, \delta)$.
\end{proposition}

\begin{remark} \label{remark:representations-with-the-same-casimir-eigenvalue}
 By Proposition \ref{prop:autovalores-casimir-mesma-esfera-centrada-em--delta}, we note that all the representations in $\widehat{G}_K$ with Casimir eigenvalue $a^2-(\delta,\delta)$ are completely described by the set
\[
S(a):= \{ \mu \in  \Gamma_G^* \; | \; V^\mu \in \widehat{G}_K \; \text{e} \; a_\mu = a  \} = \S_a(-\delta) \cap  \Gamma_G^* \; .
\]
Even the elements $\mu \in S(a)-C$ describes some representation $V^{w_\mu \cdot \mu} \in \widehat{G}_K$, where $w_\mu$ is an element of the Weyl group $W$ which sends $\mu$ to the Weyl chamber $C$.
\end{remark}

Due to Proposition \ref{prop:autovalores-casimir-mesma-esfera-centrada-em--delta}, we want to move our root system to $-\delta$, that is, we want to put a copy of the root system on a new origin $-\delta$ in a way which is similar to the process of putting a tangent space on an arbitrary point of a smooth manifold.  We will do that in the next subsection.

\,


 \subsection{The linear structure on $-\delta$} \label{subsec:The-linear-structure-on--delta}

We define $\widetilde{\a}$ to be $\a$ as sets, but $\widetilde{\a}$ is endowed with the following vector operations 
 \[
\forall H,H' \in \a , \, \forall c \in \R, \;  \left\{ \begin{array}{l}
    (H-\delta) \oplus (H'-\delta) := (H-H')-\delta   \\
    c \odot (H-\delta) :=   (cH)-\delta \\
    \llangle H-\delta , H'-\delta \rrangle := \langle H, H' \rangle \quad .
\end{array} \right.
\]

\begin{remark} \label{remark:linear-isometry-delta-deslocada}
    Let $T_H$ the translation map $H' \mapsto H'+H$. One can easily check that $T_{-\delta}:\a \to \widetilde{\a}$ is a linear isometry with inverse $T_{\delta}:\widetilde{\a} \to \a$.
\end{remark}

Remark \ref{remark:linear-isometry-delta-deslocada} give us a way to move for $\a$ to $\widetilde{\a}$ (and vice-versa). There are other induced relations between these vector spaces. For example, every $\varphi \in \End(\a)$ induces a map
\[
\widetilde{\varphi}: \widetilde{\a}\ni (H-\delta) \mapsto (\varphi(H)-\delta) \in \tilde{\a} 
\]
If we set 
\[
\widetilde{\End(\a)}:= \{ \widetilde{\varphi} \, ; \; \varphi \in \End(\a) \} \; , 
\]
then $\widetilde{\End(\a)} = \End(\widetilde{\a})$ and this ring is isomorphic to the ring $\End(\a)$. Similarly, if we set 
\[
\widetilde{O(\a)}:= \{ \widetilde{\varphi} \, ; \; \varphi \in O(\a) \} \; , 
\]
then $\widetilde{O(\a)} = O(\widetilde{\a})$.

Recall Remark \ref{remark:representations-with-the-same-casimir-eigenvalue} and define
\[
O(\widetilde{\a})_{S(a)}:= \{ \widetilde{\varphi} \in O(\widetilde{\a}) ; \; \widetilde{\varphi}(S(a)) \subset S(a) \} \; .
\]
Since $S(a)$ is a finite set and each map in $O(\widetilde{\a})$ is one-to-one, then we have the equality

\begin{equation*}
    O(\widetilde{\a})_{S(a)} = \{ \widetilde{\varphi} \in O(\widetilde{\a}) ; \; \widetilde{\varphi}(S(a)) = S(a) \} \; 
\end{equation*}
and $O(\widetilde{\a})_{S(a)}$ is a subgroup of $O(\widetilde{\a})$.

We want to show that $O(\widetilde{\a})_{S(a)}$ has finite order and that it acts transitively on $S(a)$. We do that in the following subsections.

\,


 \subsection{Representations with the same Casimir eigenvalue} \label{subsec:representations-with-the-same-Casimir-eigenvalue}

 \begin{theorem} \label{th:transitive-action}
 Let $a>0$ such that $\lambda := a^2 - (\delta, \delta)$ is a Casimir eigenvalue. The group  $O(\widetilde{\a})_{S(a)}$ is a finite subgroup of $O(\widetilde{\a})$ and it acts transitively on $S(a)$. 
 \end{theorem}
For the proof of the theorem above, see Subsection \ref{subsec:proof-theorem-representations-with-the-same-Casimir-eigenvalue}.

From the precedent theorem, we note that irreducible representations in $\widehat{G}_K$ with the same Casimir eigenvalue $\lambda = a^2-(\delta,\delta)$ are related by a transitive action of some finite algebraic symmetries in $O(\widetilde{\a})_{S(a)} \subset O(\widetilde{\a})$. By Section \ref{sec:Q8-section}, we must consider $G$-representations in the real version of the operator, while we must consider $(Q_8 \times G)$-representations if we consider its corresponding complex version.

Before we conclude our final statements and analysis, recall the configurations for the $\lambda$-eigenspaces in \ref{eq:lambda-eigenspace} and \ref{eq:lambda-eigenspace-spherical-subgroup} (the last one for a spherical subgroup).

\begin{corollary} Suppose that $M = G/K$ is a irreducible compact symmetric space with root system $\a$ satisfying $\delta \in \Gamma_G^*$ (for example, in the simply connected case, where $\Gamma_G = \I$). Then each real eigenspace of $\Delta_g$, acting on $M$, is a real irreducible representation of $G$ or it is a finite direct sum of real irreducible representations of $G$ which are related by a transitive action of a finite subgroup in $O(\widetilde{\a})$.

Also, if $\widehat{G}_K$ has only representations of real type, then each complex eigenspace of $\Delta_g$ is a complex irreducible representation of $G$ or it is a finite direct sum of real irreducible representations of $G$ which are related by a transitive action of a finite subgroup in $O(\widetilde{\a})$.

Similarly, if $\widehat{G}_K$ has any representation of complex or quaternionic type, then each complex eigenspace of $\Delta_g$ is a complex irreducible representation of $Q_8 \times G$ or it is a finite direct sum of  complex irreducible representations of $Q_8 \times G$ which are related by a transitive action of a finite subgroup in $O(\widetilde{\a})$.
\end{corollary}

Note that the panorama on the precedent corollary is satisfied for all the metrics $g$ such that $(M,g)$ is a symmetric space. In particular, this is a generic setting. 

An interesting question remains: when $K$ is not a spherical subgroup of $G$ (in particular, $(G,K)$ is not a symmetric pair), can we still obtain an analogous statement?

\,


\subsection{Proof of the Theorem \ref{th:transitive-action}} \label{subsec:proof-theorem-representations-with-the-same-Casimir-eigenvalue}
Fix $a>0$ such that $\lambda := a^2 - (\delta, \delta)$ is a Casimir eigenvalue. We recall that the Weyl group of $\a$ is denoted by $W$ and, in our context, we are assuming $\delta \in \Gamma_G^*$ (for example, when $\Gamma_G = \I$ as in the simply connected case) and $\a$ irreducible. We begin with a lemma.

\begin{lemma} \label{lemma:weyl-group}
   Let $\widetilde{W}:= \{ \widetilde{w} \in O(\widetilde{\a}) \, ; \;  w \in W \}$. Then $\widetilde{W} \subset O(\widetilde{\a})_{S(a)}$.
\end{lemma}

\begin{proof}
    Let $\mu \in S(a) = \S_a(-\delta) \cap \Gamma_G^*$. Then for each $w \in W$, we have
    \[
    \begin{array}{rcl}
      \widetilde{w}(\mu)&=&\widetilde{w}(\mu+\delta-\delta)  \\
                        &=&w(\mu+\delta)-\delta  \\
                        &=&w(\mu)+w(\delta)-\delta 
    \end{array}
    \]
    Since $W(\Gamma_G^*) \subset \Gamma_G^*$ and $\Gamma_G^*$ is a lattice, then $w(\mu)+w(\delta)-\delta \in \Gamma_G^*$. Also, since $\widetilde{w} \in O(\widetilde{\a})$, we have $\widetilde{w}(\mu) \in \S_a(-\delta)$. Thus, we conclude that $\widetilde{w}(S(a)) \subset S(a)$, that is,  $\tilde{w} \in O(\widetilde{\a})_{S(a)}$. 
\end{proof}

\begin{remark} \label{remark:W-irreducible-action}
    It is well known that the $W$-action on $\a$ is irreducible. Thus, we conclude that the $\widetilde{W}$-action on $\widetilde{\a}$ is also irreducible. 
\end{remark}

\begin{proof}[Proof of Theorem \ref{th:transitive-action}]
Choose $\mu \in S(a) = \S_a(-\delta) \cap \Gamma_G^*$. Since $a>0$, we have $\mu + \delta \neq 0$. By Remark \ref{remark:W-irreducible-action}, the set $\widetilde{W}\cdot \mu$ spans the vector space $\widetilde{\a}$. Thus, we can choose elements $ w_1, \dots, w_r \in W$, $w_1 = \id$, such that $\beta:= \{ \widetilde{w}_1 \cdot \mu , \dots, \widetilde{w}_r \cdot \mu \}$ is a basis of $\widetilde{\a}$. By Lemma \ref{lemma:weyl-group}, $\beta \subset S(a)$.

Let $\widetilde{\varphi} \in O(\widetilde{\a})_{S(a)}$. Since $\widetilde{\varphi}$ is completely characterized by its values evaluated on the basis $\beta$ and, beyond that, $S(a)$ is a finite set, then we have only finite possibilities to define $\widetilde{\varphi}$. Thus, $O(\widetilde{\a})_{S(a)}$ must be a finite set.

Now, let $\eta \in S(a)$. By the same argument in the last paragraph, we can construct $\widetilde{\varphi}\in O(\widetilde{\a})_{S(a)}$ such that $\widetilde{\varphi}(\widetilde{w_1} \cdot \mu) 
 = \widetilde{\varphi}(\mu) = \eta$, which proves that the action is transitive.

\end{proof} 


\subsection*{Acknowledgements}

This research was partially supported by Coordenação de Aperfeiçoamento de Pessoal de Nível Superior (CAPES) and partially supported by Fundação de Amparo à Pesquisa do Estado do Amazonas (FAPEAM).



\small
\bibliographystyle{apalike}


\end{document}